\documentclass[a4paper]{article}\def\zibreport{1}

\usepackage{ifthen}
\usepackage{amsmath,amsfonts,amssymb}
\usepackage{tabularx,supertabular,booktabs}
\usepackage{array,multirow,graphicx,bigdelim}
\usepackage{todonotes}
\usepackage{xspace}
\usepackage{pdflscape}
\usepackage{lscape}
\usepackage{subfig}
\usepackage {footnote}
\usepackage{footmisc}
\usepackage {tikz}
\usepackage[linesnumbered,ruled]{algorithm2e}
\usetikzlibrary{shapes,arrows, matrix,decorations}
\usetikzlibrary{decorations.pathreplacing}
\usetikzlibrary{automata}

\usepackage{siunitx}

\usepackage{url}
\usepackage[bookmarks]{hyperref}

\newcommand{\trd}{terminal-regions decomposition\xspace}

\newcommand{\myurl}[1]{\textsf{\footnotesize \url{#1}}\xspace}%
\newcommand{\stp}{{SPG}\xspace}

\newcommand{\ud}{\underline{d}}
\newcommand{\uv}{\underline{v}}
\newcommand{\udn}{$\underline{d}$-nearest }

\newcommand{\pc}{PCSTP\xspace}
\newcommand{\rpc}{RPCSTP\xspace}

\newcommand{\ppc}{I_{PC}}
\newcommand{\LP}{{LP}\xspace}

\newcommand{\with}{\ensuremath{\mid}\xspace}

\newcommand{\xp}{\tilde{x}}
\newcommand{\yp}{\tilde{y}}
\newcommand{\xq}{\tilde{x}}
\newcommand{\yq}{\tilde{y}}

\newcommand{\solver}[1]{\textsc{#1}\xspace}

\newcommand{\scipjack}{\solver{SCIP-Jack}}

\newcommand{\inp}{\textbf{Input: }}
\newcommand{\outp}{\textbf{Output: }}

\ifthenelse{\zibreport = 1}{
	\usepackage{geometry}
	\usepackage{amsthm}
	\iftrue

	\usepackage{zibtitlepage}

	\ZTPTitle{Reduction-based exact solution of prize-collecting Steiner tree problems}
	\ZTPAuthor{Daniel Rehfeldt, Thorsten Koch}
	\ZTPPreprint
	\ZTPNumber{18-55}
	\ZTPMonth{November}
	\ZTPYear{2018}
	\fi
}{}
\newtheorem{lemma}{Lemma}
\newtheorem{proposition}{Proposition}

\newtheorem{formulation}{Formulation}
\newtheorem{transf}{Transformation}
\newtheoremstyle{dotless}{}{}{\itshape}{}{\bfseries}{}{ }{}
\theoremstyle{dotless}

\begin{document}
	
	\ifthenelse{\zibreport = 1}{
		\title{Reduction-based exact solution of prize-collecting Steiner tree problems}
	}
	{
		\title{Reduction-based exact solution of prize-collecting Steiner tree problems}
	}
	
	\ifthenelse{\zibreport = 1}{
		\author{Daniel Rehfeldt\thanks{TU Berlin, 10623 Berlin, Germany}~\thanks{Zuse Institute Berlin, 14195~Berlin, Germany,  koch@zib.de, rehfeldt@zib.de}~~$\cdot$~Thorsten Koch}
		\date{}
	}
	{
		
		\author{ Daniel Rehfeldt \and Thorsten Koch
			
		}
		\institute{Zuse Institute Berlin, Takustr.~7, 14195~Berlin, Germany, \\\email{koch@zib.de, rehfeldt@zib.de}%
		}
	}
	\iftrue
	\ifthenelse{\zibreport = 1}{
		
		\zibtitlepage
		
		\thispagestyle{empty}
		\setcounter{page}{1}
	}{}
	\fi
	\maketitle
	\begin{abstract}
		The prize-collecting Steiner tree problem (\pc) is a well-known generalization of the classical Steiner tree problem in graphs, with a large number of practical applications. It attracted particular interest during the latest (11th) DIMACS Challenge and since then a number of \pc solvers have been introduced in the literature, some of which drastically improved on the best results achieved at the Challenge. The following article aims to further advance the state of the art. It introduces new techniques and algorithms for \pc, involving various forms of reductions of \pc instances to equivalent problems---which for example allows to decrease the problem size or to obtain a better IP formulation. 
		Several of the new techniques and algorithms provably dominate previous approaches.
		Further theoretical properties of the new components, such as their complexity, are discussed, and their profound interaction is described.
		Finally, the new developments also translate into a strong computational performance: the resulting exact solver outperforms all previous approaches---both in terms of run-time and solvability---and can solve formerly intractable benchmark instances from the 11th DIMACS Challenge to optimality.
	\end{abstract}
		\section{Introduction}
		\label{sec:introduction}
		
		The Steiner tree problem in graphs (SPG) is one of the fundamental ($\mathcal{NP}$-hard) combinatorial optimization problems~\cite{K72}. A well-known generalization is the prize-collecting Steiner tree problem (\pc), stated as follows: Given an undirected graph $G=(V,E)$, edge-weights $c: E \to \mathbb{Q}_{>0}$, and node-weights (or \emph{prizes}) $p: V \to \mathbb{Q}_{\geq 0}$, a tree $S = (V(S),E(S)) \subseteq G$ is required such that
		\begin{align}
			C(S) :=  \sum_{e \in E(S)} c(e) + \sum_{v \in V \setminus V(S)} p(v)
		\end{align}
		is minimized. By setting sufficiently high node weights for its terminals, each \stp instance can be reduced to a \pc. However, while the number of real-world applications of the classical Steiner tree problem in graphs is limited~\cite{SCIPJACKMPC}, the \pc entails many practical applications, which can be found in various areas, for instance in the design of telecommunication networks~\cite{IvanaL2004}, electricity planning~\cite{PCelectric}, computational biology~\cite{IOSS02}, or geophysics~\cite{Schmidt2015}.
	
		The \pc has been extensively discussed in the literature, see e.g.~\cite{Bie93,Can01,Joh00,Lju06}. Moreover, many exact and heuristic solving approaches have been suggested. The problem attracted particular interest in the wake of the 11th DIMACS Challenge~\cite{DimacsWeb} in December 2014---dedicated to Steiner tree problems---where the \pc categories could boast the most participants by far. Furthermore, in the last two years a number of solvers for the \pc have been described in the literature~\cite{AKHMEDOV201618,Fischetti2017,FuH17,SCIPJACKMPC,Leitner18,PCswap,sun2018classical}, some of which, in particular~\cite{Leitner18}, could drastically improve on the best results achieved at the DIMACS Challenge---being able to not only solve many instances orders of magnitude faster, but also to solve a number of instances for the first time to optimality. Exact approaches for \pc are usually based on branch-and-bound or branch-and-cut~\cite{Fischetti2017,SCIPJACKMPC}, include specialized (primal and sometimes dual) heuristics~\cite{lju04,Leitner18}, and make use of various preprocessing methods to reduce the problem size~\cite{Lju06,RehfeldtKochMaher16}.
		
		 This article introduces new techniques and algorithms for solving \pc, most of which are based on, or result in reductions of the \pc 
		 to equivalent problems---these problems can be PCSTPs itself, but can also be from different problem classes. The reductions can for example decrease the problem size or allow to obtain a stronger IP formulation. Moreover, several of the new methods provably dominate previous approaches.
		 While some of the techniques require to solve $NP$-hard subproblems (not yet described in the literature), the underlying concepts allow to design empirically efficient heuristics.
		Practically also the integration of the new methods into an exact solver will be described and computational experiments on a large number of benchmark instances will be presented---along with a comparison with a state-of-the-art solver.
		 What sets the new techniques apart from other approaches for combinatorial optimization problems is the deep and intricate interaction of the individual components combined with their wide applicability within a branch-and-cut framework---from preprocessing and probing to IP formulation and separation to heuristics, domain propagation and branching.
		 While interaction between solution techniques and multiple usability has been described in state-of-the-art solvers for problems such as the traveling salesman problem~\cite{TSP} or (even more pronounced) the \stp~\cite{Polzin04}, we are not aware of any other approach for which such a broad impact on the overall solving procedure and profound interrelation can be achieved by such a, comparatively, small number of techniques.   
		 We would also like to point out that the newly developed software has been integrated into the academic Steiner tree framework~\scipjack~\cite{SCIPJACKMPC} and will be made publicly available as part of its next major release.
		
		Finally, while it will be shown that one set of methods is also directly applicable for the \stp, one can furthermore extend several of the presented techniques and algorithms to related combinatorial optimization problems such as the node-weighted Steiner tree, or the maximum-weight connected subgraph problem.
		
		\subsection{Preliminaries and notation}
		Throughout this article it will be presupposed that a \pc instance $\ppc = (V,E,c,p)$ is given such that $(V,E)$ is connected; otherwise one can optimize each connected component separately. For a graph $G$ we denote its vertices by $V(G)$ and its edges by $E(G)$; similarly, for a finite walk $W$ we denote the set of vertices and the set of edges it contains by $V(W)$ and $E(W)$. We call $T_p := \{t_1,t_2,...,t_s \} := \{v \in V \with p(v) > 0 \}$ the set of \emph{potential terminals}~\cite{Leitner18}.
		
		 By $d(v_i,v_j)$ we denote the distance of a shortest path (with respect to $c$) between vertices $v_i, v_j \in V$.
	 Similarly, $\ud(v_i,v_j)$ is defined as the distance between $v_i$ and $v_j$ in the graph induced by $V \setminus (T_p \setminus \{v_i, v_j\})$~\cite{RehfeldtKochMaher16}. To each vertex $v_i \in V \setminus T$, the $k$ \udn potential terminals will be denoted by $\uv_{i,1}, \uv_{i,2},...,\uv_{i,k}$ (ties are broken arbitrarily). 
		 Moreover, for any function $x: M \mapsto \mathbb{Q}$ with $M$ finite, and any $M' \subseteq M$ define $x(M') := \sum_{i \in M'} x(i)$. 
		For $U \subseteq V$ define $\delta(U):=\{ \{u,v\} \in E \with u\in U, v\in V\setminus U\}$;
	for a directed graph $D = (V,A)$ define $\delta^+(U):=\{(u,v)\in A \with u\in U, v\in V\setminus U\}$ and $\delta^-(U):= \delta^+(V \setminus U)$. We also write $\delta_G$ or $\delta^+_D, \delta^-_D$ to distinguish the underlying graph.  
		For an IP formulation $F$ we denote 
		the optimal objective value and the set of feasible points of its \LP relaxation by $v_{LP}(F)$ and $\mathcal{P}_{LP}(F)$, respectively.
			
				\section{Reductions within the problem class}
				\label{sec:reduction}
			    The reductions described in the following aim to reduce a given instance to a smaller one of the same problem class. Several articles have addressed such techniques for the \pc, e.g.~\cite{Leitner18,Lju06,RehfeldtKochMaher16,Uchoa06}, but most are dominated by the methods described in the following.
			    The new methods will not only be employed for classical preprocessing, but also throughout the entire solving process, e.g. for domain propagation or within heuristics. 
			
				\subsection{Taking short walks} 
					\label{sec:reduction:alt}
			The following approach uses a new, walk-based, distance function. It generalizes the bottleneck distance concept that was the central theme of~\cite{Uchoa06}. 
						Let $v_i, v_j \in V$. A finite walk $W = (v_{i_1},e_{i_1}, v_{i_2},e_{i_2},...,e_{i_r}, v_{i_r})$ with $v_{i_1} = v_i$ and $v_{i_r} = v_j$ will be called \textit{prize-constrained $(v_i, v_j)$-walk} if no $v \in T_p \cup \{v_i,v_j\}$ is contained more than once in $W$. For any $k,l \in \mathbb{N}$ with $1 \leq k \leq l \leq r$ define the subwalk $W(v_{i_k}, v_{i_l}) := (v_{i_k}, e_{i_k},v_{i_{k+1}}, e_{i_{k+1}},...,e_{i_l}, v_{i_l})$; note that $W(v_{i_k}, v_{i_l})$ is again a prize-constrained walk. 
					Furthermore, define the \textit{prize-collecting cost} of $W$ as 
							\begin{equation}
						c_{pc}(W) :=	\sum_{e \in E(W)} c(e) - \sum_{v \in V(W) \setminus \{ v_{i}, v_{j} \} } p(v).
							\end{equation}
				Thereupon, define the \textit{prize-constrained length} of $W$ as
				\begin{equation}
				l_{pc}(W) := \max\{c_{pc}(W(v_{i_k},v_{i_l})) \with 1 \leq k \leq l \leq r,~  v_{i_k},v_{i_l} \in T_p \cup \{v_i,v_j\}\}.
				\end{equation}
				Intuitively, $l_{pc}(W)$ provides the cost of the least profitable subwalk of $W$. This measure will in the following be useful to bound the cost of connecting two disjoint trees that contain the first and the last vertex of $W$, respectively.
			Finally, we denote the set of all prize-constrained $(v_i, v_j)$-walks by $\mathcal{W}_{pc}(v_i,v_j)$ and define the \textit{prize-constrained distance} between $v_i$ and $v_j$ as
				\begin{equation}
				d_{pc}(v_i,v_j) := \min\{  l_{pc}(W') \with W' \in \mathcal{W}_{pc}(v_i,v_j)\}. 
				\end{equation}
				Note that $d_{pc}(v_i,v_j) = d_{pc}(v_j,v_i)$ for any $v_i,v_j \in V$.
				By using the prize-constrained distance one can formulate a reduction criterion that dominates the \emph{special distance} test from~\cite{Uchoa06}---it identifies (and allows to delete) all edges found by~\cite{Uchoa06} and can (and usually does) identify further ones:
				\begin{proposition}
					\label{prop:sd}
					Let $\{v_i, v_j \} \in E$. If
					\begin{equation}
							\label{prop:sd:1}
					c(\{v_i, v_j \}) > d_{pc}(v_i,v_j)
					\end{equation}
					is satisfied, then $\{v_i, v_j \}$ cannot be contained in any optimal solution.
				\end{proposition}
					\begin{proof}
							Let $S$ be a tree with $\{v_i, v_j \} \in E(S)$. Further, let  $W = (v_{i_1},e_{i_1},...,e_{i_r}, v_{i_r})$ be a prize-constrained $(v_i,v_j)$-walk with $l_{pc}(W) = d_{pc}(v_i,v_j)$. Remove $\{v_i, v_j\}$ from $S$ to obtain two new trees. Of these two trees denote the one that contains $v_i$ by $S_i$ and the other (containing $v_j$) by $S_j$.
							 Define $b := \min\{k \in \{1,...,r\} \with v_{i_k} \in V(S_j) \}$ and		 
							 $a := \max\{k \in \{1,...,b\} \with v_{i_{k}} \in V(S_i)\}$. 
							 Further, define $x := \max\{k \in \{1,...,a\} \with v_{i_{k}} \in T_p \cup \{v_i\} \}$ and 
							 $y := \min\{k \in \{b,...,r\} \with v_{i_k} \in T_p \cup \{v_j\} \}$. By definition, $x \leq a < b \leq y$ and furthermore:
							 \begin{equation}
							 		\label{proof:sd:1}
							 	c_{pc}(W(v_{i_a}, v_{i_b})) \leq c_{pc}(W(v_{i_x}, v_{i_y})).
							 \end{equation}
						 Reconnect $S_i$ and $S_j$ by $W(v_{i_a}, v_{i_b}))$, which yields a new connected subgraph $S'$. If $S'$ is not a tree, make it one by removing redundant edges, without removing any node (which can only decrease $C(S')$). For this tree it holds that:
						 						\begin{align*}
											       C(S') &\leq C(S) + c_{pc}(W(v_{i_a}, v_{i_b})) - c(\{v_i, v_j \}) \\
							 							 &\stackrel{\eqref{proof:sd:1}}{\leq}  C(S) + c_{pc}(W(v_{i_x}, v_{i_y})) - c(\{v_i, v_j \}) \\
							 							 &\leq  C(S) + l_{pc}(W) - c(\{v_i, v_j \}) \\
							 							 &=  C(S) + d_{pc}(v_i,v_j) - c(\{v_i, v_j \}) \\
							 							 &\stackrel{\eqref{prop:sd:1}}{<}  C(S).
							 				\end{align*}
							 		Because of $C(S') < C(S)$ no optimal solution can contain $\{v_i, v_j \}$.
								\end{proof}
		Since computing the Steiner bottleneck distance is already $\mathcal{NP}$-hard~\cite{Uchoa06}, it does not come as a surprise that the same holds for $d_{pc}$ (which can be shown in the same way).
				 However, the definition of $d_{pc}$ allows to design a simple algorithm for finding upper bounds that yields empirically strong results. 
				The method is an extension of Dijkstra's algorithm~\cite{Dijkstra1959} (with priority queue), and is run from both endpoints of an edge $e = \{v_i,v_j\}$ that one attempts to delete. We sketch the procedure for endpoint $v_i$: 
				Let $dist(v)$ be the distance value of Dijkstra's algorithm for each $v \in V$, initialized with $dist(v) := \infty$ for all $v \in V \setminus \{v_i\}$ and $dist(v_i) := 0$. 
				Start Dijkstra's algorithm from $v_i$, but apply the following modifications: 
				First, do not update vertex $v_l$ from vertex $v_k$ if $dist(v_k) + c(\{v_k,v_l\}) \geq c(e)$. 
				Second, for $t \in T_p \setminus \{v_i\}$ set $dist(t)$ (as computed by Dijkstra) to $\max\{dist(t) - p(t), 0\}$ before inserting $t$ into the priority queue or when updating $t$ (if it already is in the priority queue). Third, all $v \in V \setminus (T_p \cup \{v_i,v_j\})$ can be reinserted into the priority queue after they have been removed. Fourth, exit if $v_j$ is visited. If $dist(v_j) < c(e)$, we can remove $e$.
				Bounds on the number of visited edges are set to limit the run time. Note that both Proposition~\ref{prop:sd} and the heuristic can be extended to the case of equality.
				
				By using the prize-constrained distance, further reduction tests from the literature can be strengthened, such as \textit{Non-Terminal of Degree k}~\cite{Uchoa06}; a related walk-based concept introduced in Section~\ref{sec:probred} allows to strengthen additional methods such as \textit{Nearest Vertex}~\cite{RehfeldtKochMaher16}, or \textit{Short Links}~\cite{RehfeldtKochMaher16}.

				\subsection{Using bounds}
				
				{Bound-based reductions} techniques identify edges and vertices for elimination by examining whether they induce a lower bound that exceeds a given upper bound~\cite{Dui93,Polzin04}. 
				For the subsequent reduction techniques of this type we introduce the following concept: a \emph{\trd} of $\ppc$ is a partition $H = \big\{ H_{t} \subseteq V \with T_p \cap H_{t} = \{t\}\big \}$ of $V$ such that for each $t \in T_p$ the subgraph induced by $H_{t}$ is connected. 
				Define for all $t \in T_p$
				\begin{equation}
					\label{bnd1}
					r_H^{pc}(t) := \min \big \{ p(t), \min\{ d(t,v) \with v \notin H_{t} \} \big \}.
				\end{equation}
				The \trd concept generalizes the Voronoi decomposition for the \pc introduced in~\cite{RehfeldtKochMaher16} and can easily be extended to \stp by using $\min\{ d(t,v) \with v \notin H_{t} \}$ instead of $r_H^{pc}(t)$---which generalizes the \stp Voronoi concept~\cite{Polzin04}. In~\cite{MWCS18} we introduced a related concept for the maximum-weight connected subgraph problem. For ease of presentation assume $r_H^{pc}(t_i) \leq r_H^{pc}(t_j)$ for $1 \leq i < j \leq s$.
				\begin{proposition}
					\label{prop:bnd1}
					Let $v_i \in V \setminus T_p$. If there is an optimal solution $S$ such that $v_i \in V(S)$, then a lower bound on $C(S)$ is defined by
					\begin{equation}
						\label{prop:bnd1:1}
					 \ud(v_i, \uv_{i,1}) + \ud(\uv_i, \uv_{i,2}) + \sum_{k = 1}^{s - 2} r_H^{pc}(t_k).
					\end{equation}
				\end{proposition}   
				The proposition can be proven similarly to the Voronoi decomposition results~\cite{RehfeldtKochMaher16}, see Appendix~\ref{app:bnd}.
				Each vertex $v_i \in V \setminus T_p$ with the property that the affiliated lower bound~\eqref{prop:bnd1:1} exceeds a known upper bound can be eliminated. Moreover, if a solution $S$ corresponding to the upper bound is given and $v_i \notin V(S)$, one can also eliminate $v_i$ if the lower bound stated in Proposition~\ref{prop:bnd1} is equal to $C(S)$.
				 A result similar to Proposition~\ref{prop:bnd1} can be formulated for edges of an optimal solution.
			Finally, the following proposition allows to \emph{pseudo-eliminate}~\cite{Dui93} vertices, i.e., to delete a vertex and connect all its adjacent vertices by new edges.	
				
				\begin{proposition}
					\label{prop:bnd3} 
					Let $v_i \in V \setminus T_p$. If there is an optimal solution $S$ such that $\delta_S(v_i) \geq 3$, then a lower bound on $C(S)$ is defined by
					\begin{equation}
						\label{prop:bnd3:1}
						\ud(v_i, \uv_{i,1}) + \ud(v_i, \uv_{i,2}) + \ud(v_i, \uv_{i,3})+ \sum_{k = 1}^{s - 3} r_H^{pc}(t_k).
					\end{equation}
				\end{proposition} 
					To efficiently apply Proposition~\ref{prop:bnd1}, one would like to maximize~\eqref{prop:bnd1:1}---and for Proposition~\ref{prop:bnd3} to maximize~\eqref{prop:bnd3:1}. However, as shown in Appendix~\ref{app:pvd}, one obtains
					\begin{proposition}
							\label{proposition:pvdnp}
						Given a $v_i \in V \setminus T_p$, finding a \trd that maximizes~\eqref{prop:bnd1:1} is $\mathcal{NP}$-hard. The same holds for~\eqref{prop:bnd3:1}.
						\end{proposition} 
					
				Thus, to compute a \trd a heuristic approach will be used; still, the flexibility of the terminal regions concept allows to design a heuristic that computes bounds
				 that are at least as good as those provided by the Voronoi decomposition method~\cite{RehfeldtKochMaher16}, and empirically often better---allowing for significantly stronger graph reductions. 
		
				
						\iftrue
						 Figure~\ref{fig:trd} depicts a \pc, a corresponding Voronoi  decomposition as described in~\cite{RehfeldtKochMaher16}---which is in fact a particular \trd---and an alternative \trd, found by our heuristic. The second \trd yields a stronger lower bound than the Voronoi decomposition. For the filled vertex the lower bounds are $10$ and $11$, respectively ($10$ is also the optimal solution value for the instance).
			
				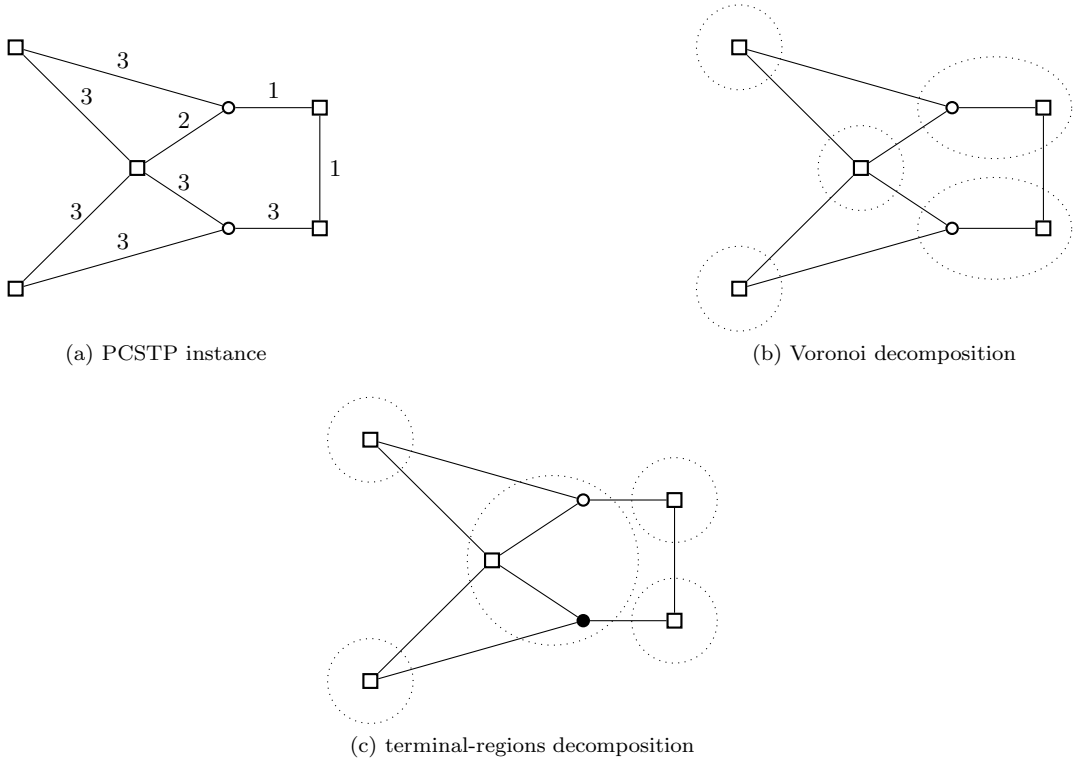
\begin{figure}[bt]
					\centering
					\subfloat[\pc instance]{
						\begin{tikzpicture}[scale=1.6]
						  \tikzstyle{terminal} = [draw, thick, minimum size=0.6, inner sep=2.6pt];
						  \tikzstyle{steiner} = [circle, draw, thick, minimum size=0.2, inner sep=1.5pt];
						   \tikzstyle{steinerf} = [circle, draw, thick, minimum size=0.4, inner sep=2.4pt];

						  \node[terminal, label=above:{}]  (t1) at (0,0) {};
						    \node[terminal, label=above:{}]  (t2) at (-1,1) {};
						    \node[terminal, label=above:{}]  (t3) at (-1,-1) {};

						        \node[steiner, label=above:{}]  (s1) at (0.75,0.5) {};
						        \node[steiner, label=above:{}]  (s2) at (0.75,-0.5) {};
						        
						            \node[terminal, label=above:{}]  (t4) at (1.5,0.5) {};
						            \node[terminal, label=above:{}]  (t5) at (1.5,-0.5) {};

						  \draw[dotted, white] (-1.0, 1.0) circle (10pt);
						  \draw[dotted, white] (-1, -1) circle (10pt);

						  \draw[] (t1) -- (t2) node [midway, above right=-2.9pt] {\small $3$};
						  \draw[] (t1) -- (t3) node [midway, above=-0.0pt] {\small $3$};
						  
						  \draw[] (s1) -- (t2) node [midway, above=-0.0pt] {\small $3$};
						  \draw[] (s2) -- (t3) node [midway, above=-0.0pt] {\small $3$};
						  \draw[] (t1) -- (s1) node [midway, above=-0.0pt] {\small $2$};
						  \draw[] (t1) -- (s2) node [midway, above=-0.0pt] {\small $3$};
						  \draw[] (s1) -- (t4) node [midway, above=-0.0pt] {\small $1$};
						  \draw[] (s2) -- (t5) node [midway, above=-0.0pt] {\small $3$};
						  \draw[] (t4) -- (t5) node [midway, right=-0.0pt] {\small $1$};
						\end{tikzpicture}
						\label{fig:trd1}
					}
					\hfill
					\subfloat[Voronoi decomposition]{
	\begin{tikzpicture}[scale=1.6]
	  \tikzstyle{terminal} = [draw, thick, minimum size=0.6, inner sep=2.6pt];
	  \tikzstyle{steiner} = [circle, draw, thick, minimum size=0.2, inner sep=1.5pt];
	   \tikzstyle{steinerf} = [circle, draw, thick, minimum size=0.4, inner sep=2.4pt];

	  \node[terminal, label=above:{}]  (t1) at (0,0) {};
	    \node[terminal, label=above:{}]  (t2) at (-1,1) {};
	    \node[terminal, label=above:{}]  (t3) at (-1,-1) {};

	        \node[steiner, label=above:{}]  (s1) at (0.75,0.5) {};
	        \node[steiner, label=above:{}]  (s2) at (0.75,-0.5) {};
	        
	            \node[terminal, label=above:{}]  (t4) at (1.5,0.5) {};
	            \node[terminal, label=above:{}]  (t5) at (1.5,-0.5) {};

	\draw[dotted] (0.0, 0.0) circle (10pt);
	\draw[dotted] (-1.0, 1.0) circle (10pt);
	\draw[dotted] (-1, -1) circle (10pt);
	\draw[dotted] (1.1, 0.5)  ellipse (18pt and 12pt);
	\draw[dotted] (1.1, -0.5) ellipse (18pt and 12pt);

	  \draw[] (t1) -- (t2) node [midway, above=-0.0pt] {\small };
	  \draw[] (t1) -- (t3) node [midway, above=-0.0pt] {\small };
	  
	  \draw[] (s1) -- (t2) node [midway, above=-0.0pt] {\small };
	  \draw[] (s2) -- (t3) node [midway, above=-0.0pt] {\small };
	  \draw[] (t1) -- (s1) node [midway, above=-0.0pt] {\small };
	  \draw[] (t1) -- (s2) node [midway, above=-0.0pt] {\small };
	  \draw[] (s1) -- (t4) node [midway, above=-0.0pt] {\small };
	  \draw[] (s2) -- (t5) node [midway, above=-0.0pt] {\small };
	  \draw[] (t4) -- (t5) node [midway, right=-0.0pt] {\small };
	\end{tikzpicture}

						\label{fig:trd2}
					}
					\hfill
					\subfloat[\trd]{
					\begin{tikzpicture}[scale=1.6]
					  \tikzstyle{terminal} = [draw, thick, minimum size=0.6, inner sep=2.6pt];
					  \tikzstyle{steiner} = [circle, draw, thick, minimum size=0.2, inner sep=1.5pt];
					   \tikzstyle{steinerf} = [circle, draw, thick, minimum size=0.4, inner sep=2.4pt];

					  \node[terminal, label=above:{}]  (t1) at (0,0) {};
					    \node[terminal, label=above:{}]  (t2) at (-1,1) {};
					    \node[terminal, label=above:{}]  (t3) at (-1,-1) {};

					        \node[steiner, label=above:{}]  (s1) at (0.75,0.5) {};
					        \node[steiner, label=above:{}, fill]  (s2) at (0.75,-0.5) {};
					        
					            \node[terminal, label=above:{}]  (t4) at (1.5,0.5) {};
					            \node[terminal, label=above:{}]  (t5) at (1.5,-0.5) {};

					\draw[dotted] (0.5, 0.0) circle (20pt);
					\draw[dotted] (-1.0, 1.0) circle (10pt);
					\draw[dotted] (-1, -1) circle (10pt);
					\draw[dotted] (1.5, 0.5) circle (10pt);
					\draw[dotted] (1.5, -0.5) circle (10pt);

					  \draw[] (t1) -- (t2) node [midway, above=-0.0pt] {\small };
					  \draw[] (t1) -- (t3) node [midway, above=-0.0pt] {\small };
					  
					  \draw[] (s1) -- (t2) node [midway, above=-0.0pt] {\small };
					  \draw[] (s2) -- (t3) node [midway, above=-0.0pt] {\small };
					  \draw[] (t1) -- (s1) node [midway, above=-0.0pt] {\small };
					  \draw[] (t1) -- (s2) node [midway, above=-0.0pt] {\small };
					  \draw[] (s1) -- (t4) node [midway, above=-0.0pt] {\small };
					  \draw[] (s2) -- (t5) node [midway, above=-0.0pt] {\small };
					  \draw[] (t4) -- (t5) node [midway, right=-0.0pt] {\small };
					\end{tikzpicture}

						\label{fig:trd3}
					}
					\caption{Illustration of a \pc instance (a), a Voronoi decomposition (b), and a second \trd (c). Potential terminals are drawn as squares. All potential terminals have a prize of $5$.
						If an upper bound less than $11$ is known, the filled vertex in (c) can be deleted by means of the \trd depicted in (c), but not by means of the Voronoi decomposition, unless also an optimal solution tree is given.}
					\label{fig:trd}
				\end{figure}  
				\fi

			\section{Changing the problem class}
			\label{sec:probred}
			
		 	A cornerstone of the reduction approach described in this section is the \textit{Steiner arborescence problem} (\textit{SAP}), which is defined as follows:
			Given a directed graph $D=(V,A)$, costs $c: A \to
			\mathbb{Q}_{\geq  0}$, a set $T \subseteq V$ of \emph{terminals} and a root $r \in T$, a directed tree (arborescence) $S \subseteq D$ of minimum cost $\sum_{a \in A(S)} c(a)$ is required such that for all $t \in T$ the tree $S$ contains a directed path from $r$ to $t$. 
			Associating with each $a \in A$ a variable $x(a)$ that indicates whether $a$ is contained in a solution ($x(a) = 1$) or not ($x(a) = 0$), one can state the well-known \textit{directed cut} (\textit{DCut}) formulation~\cite{Won84} for an SAP $(V,A,T,c,r)$:
			\begin{formulation}{Directed cut (DCut)}
				\label{form:dcut}
				\begin{eqnarray}
					\textrm{min}\quad {c}^T x\\
					\label{form:dcut:1}
					x(\delta^-(U))&\geq& 1 \hspace{22mm}\text{for all }\, U\subset V, ~ r \notin U, U \cap T \neq \emptyset, \\
					\label{form:dcut:2}
					x(a)&\in& \{0,1\} \hspace{15.1mm}\text{for all } a \in A.
				\end{eqnarray}
			\end{formulation}
			In~\cite{Won84} also a dual-ascent algorithm for $DCut$ was introduced that allows to quickly compute empirically strong lower bounds. Dual-ascent is one of the reasons it has proven advantageous to transform undirected problems such as the \stp to SAP~\cite{KochMartin1998,Polzin04} and use $DCut$, but furthermore such transformations can provide stronger LP relaxations, both theoretically and practically~\cite{Dui93,G93}.
		For the \pc such a transformation can be stated as follows~\cite{RehfeldtKoch2018}:
			\begin{transf}[\pc to SAP]~
				\label{transf:PC}
				
				~\inp{\pc $(V,E,c,p)$} 
				
				~\outp{SAP $(V',A',T',c',r')$}
				
				\begin{enumerate}
					
					\item \label{transf:PC:2} Set $V' := V$, $A':=\lbrace (v,w) \in V' \times V' \with \lbrace v,w \rbrace \in E \rbrace$, and $c': A' \to \mathbb{Q}_{\geq 0}$ with $~c'(a) := c({\{ v,w \} }) $ for $a = (v,w) \in A'$; define $M:= \sum_{t \in T_p} p(t)$.
					\item Add vertices $r'$ and $v_0'$ to $V'$.
					\item For each $i \in \{ 1,...,s\}$:
					\begin{enumerate}
						\item add arc $(r',t_i)$ of weight $M$ to $A'$;
						\item add node $t_i'$ to $V'$;
						\item add arcs $(t_i,v_0')$ and $(t_i,t_i')$ to $A'$, both being of weight $0$;
						\item add arc $(v_0',t'_i)$ of weight $p({t_i})$ to $A'$.
					\end{enumerate}
					\item Define set of terminals $T':=\{t'_1,...,t'_s\} \cup \{r'\}$.
					\item \textbf{Return} $(V',A',T',c',r')$.
				\end{enumerate}
			\end{transf}
				The underlying idea of the transformation is to add a new terminal $t_i'$ for each original potential terminal $t_i$ and provide additional arcs that allow to connect $t_i'$ from any original potential terminal $t_j$ with cost $p(t_j)$. Note that one can use a, similar, simplified transformation if one adds an additional constraint to the resulting SAP~\cite{SCIPJACKMPC}. However, besides providing a ``pure'' SAP, Transformation~\ref{transf:PC} allows to directly apply dual-ascent. The following results still hold if Transformation~\ref{transf:PC} is replaced by the simpler version.
				
			Each optimal solution to the SAP obtained from Transformation~\ref{transf:PC} can be transformed to an optimal solution to the original \pc. For $\ppc = (V,E,c,p)$ one can therefore define the following formulation, which uses the SAP $(V',A',T',c',r')$ obtained from  applying Transformation~\ref{transf:PC} on $\ppc$:
		   \begin{formulation}{Transformed prize-collecting cut (PrizeCut)}
						\label{form:PrizeCut}
			\begin{eqnarray}
				 \quad&& \textrm{min}\quad {c'}^T x - M \\
			  &&\quad\quad	x \text{ satisfies } \eqref{form:dcut:1},\eqref{form:dcut:2}~~~~\\
			  && \label{form:PrizeCut:2}	y(\{v_i,v_j\}) = x((v_i,v_j) + x((v_j,v_i))     \hspace{8.0mm}\text{for all } \{v_i,v_j\} \in E \\
			  && \label{form:PrizeCut:3}	\quad\quad~~	y(e)\in \{0,1\} \hspace{33.2mm}\text{for all } e \in E.
			\end{eqnarray}
			\end{formulation}
	The $y$ variables correspond to the solution to $\ppc$; note that removing them does not change the optimal solution value, neither that of the LP relaxation. 
		To avoid adding an artificial root (which entails \emph{big M} constants and symmetry) in the transformation to SAP, one can attempt to identify vertices that are part of all optimal solutions. 
		To this end, let $v_i, v_j \in V$, and let $W$ be a {prize-constrained $(v_i, v_j)$-walk} (as defined in Section~\ref{sec:reduction}). Define the \textit{left-rooted prize-constrained length} of $W$ as:
			\begin{equation}
				l^-_{pc}(W) := \max\{c_{pc}(W(v_{i},v_{i_k})) \with v_{i_k} \in V(W) \cap (T \cup \{v_j\})  \}.
			\end{equation}
			Furthermore, define the \textit{left-rooted prize-constrained $(v_i, v_j)$-distance} as:
			\begin{equation}
			\label{leftsided}
				d^-_{pc}(v_i,v_j) := \min\{  l^-_{pc}(W') \with W' \in \mathcal{W}_{pc}(v_i,v_j)   \}. 
			\end{equation}
			Note that in general $d^-_{pc}$ is not symmetric. Definition~\eqref{leftsided} gives rise to
			\begin{proposition}
				\label{prop:ebd}
				Let $v_i, v_j \in V$. If
				\begin{equation}
					p(v_i) > d^-_{pc}(v_i,v_j)
				\end{equation}
				is satisfied, then every optimal solution that contains $v_j$ also contains $v_i$.
			\end{proposition}
				\begin{proof}
												Let $S$ be a tree with $v_j \in V(S)$ and $v_i \notin V(S)$. Further, let  $W = (v_{i_1},e_{i_1},...,e_{i_r}, v_{i_r})$ be a prize-constrained $(v_i,v_j)$-walk with $l^-_{pc}(W) = d^-_{pc}(v_i,v_j)$ and define $a := \min\{k \in \{1,...,r\} \with v_{i_k} \in V(S)  \}$ and $b := \min\{k \in \{a,a+1,...,r\} \with v_{i_k} \in T_p \cup \{v_j\} \}$. Note that
														\begin{equation}
														\label{proof:ebd:1}
														c_{pc}(W(v_i, v_{i_a})) \leq c_{pc}(W(v_i, v_{i_b})). 
													    \end{equation}
													Add the subgraph corresponding to $W(v_i,v_{i_a})$ to $S$, which yields a new connected subgraph $S'$. If $S'$ is not a tree, make it one by removing redundant edges, without removing any node (which can only decrease $C(S')$). It holds that:
													\begin{align*}
														C(S') &~\leq ~C(S) + c_{pc}(W( v_i, v_{i_a})) - p(v_i) \\
														 &\stackrel{\eqref{proof:ebd:1}}{\leq} C(S) + c_{pc}(W( v_i, v_{i_b})) - p(v_i) \\
														 &~\leq  ~C(S) + l^-_{pc}(W) - p(v_i) \\
														 &~=  ~C(S) + d^-_{pc}(v_i,v_j) - p(v_i) \\
														 &~\stackrel{\eqref{prop:ebd}}{<} C(S).
													\end{align*}
													The relation $C(S') < C(S)$ implicates that an optimal solution that contains $v_j$ also contains $v_i$.
												\end{proof}
			As for $d_{pc}$, computing $d^-_{pc}$ is $\mathcal{NP}$-hard. However, one can devise a simple algorithm for finding upper bounds similar to the one for $d_{pc}$:
			Let $t_0 \in T_p$. The following adaptation of Dijkstra's algorithm provides a set of vertices $\bar{T}_{t_0}$ such that $d^-_{pc}(t_0, v) < p(t_0)$ for all $v \in \bar{T}_{t_0}$. Initialize $dist(v) := \infty$ for all $v \in V \setminus \{t_0\}$, and $dist(t_0) := 0$. Start Dijkstra's algorithm from $t_0$, but apply the following modifications:  
							First, do not update vertex $v_j$ from vertex $v_i$ if $dist(v_i) + c(\{v_i,v_j\}) \geq p(t_0)$.
							Second, for $t \in T_p \setminus \{t_0\}$ set $dist(t)$ (as computed by Dijkstra) to $dist(t) - p(t)$ before inserting $t$ into the priority queue or when updating $t$ (if it already is in the priority queue). Third, all $v \in V \setminus T_p$ can be reinserted into the priority queue after they have been removed. Finally, define 
							$\bar{T}_{t_0} := \{v \in V \with dist(v) < p(t_0)\}$.
			
			By using LP information, this heuristic can be combined with Transformation~\ref{transf:PC} to obtain a criterion for potential terminals to be part of all optimal solutions. First, note that if a separation algorithm or dual-ascent is applied, one obtains reduced costs for an LP relaxation of $DCut$ that contains only a subset of constraints~\eqref{form:dcut:1}. Second, observe that given an SAP $I'$ obtained from $\ppc$ with corresponding optimal solutions $S'$ and $S$, for $t_i \in T_p$ it holds that $t_i \in V(S)$ if and only if $(v_0',t_i') \notin A'(S')$. As a consequence one obtains
			\begin{proposition}
				\label{prop:imp}
				Consider $(V',A',T',c',r')$ obtained by applying Transformation~\ref{transf:PC} on $\ppc$. Let $\tilde{\mathcal{U}} \subseteq \{ U\subset V' \with r' \notin U, U \cap T' \neq \emptyset \}$ and let $\tilde{L}$ be the objective value and $\tilde{c}$ the reduced costs of an optimal solution to the \LP:
									\begin{eqnarray}
												\textrm{min}\quad {c'}^T x - M\\
												x(\delta^-(U))&\geq& 1 \hspace{22mm}\text{for all }\, U \in \tilde{\mathcal{U}}, \\
												x(a)&\in& [0,1] \hspace{17.0mm}\text{for all } a \in A'. 
									\end{eqnarray}
				 Moreover, let $U$ be an upper bound on the cost of an optimal solution to $\ppc$.
				Finally, let $t_i \in T_p$ and let $\bar{T}_i \subseteq T_p$ such that $V(S) \cap \bar{T}_i \neq \emptyset \Rightarrow t_i \in V(S)$ for each optimal solution $S$ to $\ppc$. If
				\begin{equation}
				\label{prop:imp:1}
					\sum_{j \with t_j \in \bar{T}_i} \tilde{c}((v_0',t'_j)) + \tilde{L} > U
				\end{equation}
				holds, then $t_i$ is part of all optimal solutions to $\ppc$.
			\end{proposition}
			
			If a $t_i \in T_p$ has been shown to be part of all optimal solutions, by building $\bar{T}_i$ with Proposition~\ref{prop:ebd} and using~\eqref{prop:imp:1}, Proposition~\ref{prop:ebd} can again be applied---to directly identify further $t_j \in T_p$ that are part of all optimal solutions by using the condition $p(t_j) > d^-_{pc}(t_j,t_i)$. Identifying such \emph{fixed} terminals can considerably improve the strength of the techniques described in Section~\ref{sec:reduction}, which usually leads to further graph reductions and the fixing of additional terminals.
			Moreover, in this case the \pc can be reduced to a \textit{rooted prize-collecting Steiner tree problem} (\rpc)\footnote{Note that in the literature it is more common to denote only problems with exactly one fixed terminal as rooted prize-collecting Steiner tree problem.}, which incorporates the additional condition that a non-empty set $T_f \subseteq V$ of \textit{fixed terminals} needs to be part of all feasible solutions. Assuming $T_p \setminus T_f = \{t_1,...,t_z\}$, we introduce the following simple transformation:
		
			
			\begin{transf}[\rpc to SAP]~
				\label{transf:RPC}
				
				~\inp{\rpc $(V,E,T_f,c,p)$ and $t_p, t_q \in T_f$}
				
				~\outp{SAP $(V',A',T',c',r')$}
				
				\begin{enumerate}
					
					\item \label{transf:RPC:2} Set $V' := V$, $A':=\lbrace (v,w) \in V' \times V' \with \lbrace v,w \rbrace \in E \rbrace$, $c':= c$, $r' := t_q$.
					\item For each $i \in \{1,...,z\}$:
					\begin{enumerate}
						\item add node $t_i'$ to $V'$,
						\item add arc $(t_i,t_i')$ of weight $0$ to $A'$,
						\item add arc $(t_p,t_i')$ of weight $p(t_i)$ to $A'$.
					\end{enumerate}
					\item Define set of terminals $T':=\{t'_1,...,t'_z\} \cup T_f$.
					\item \textbf{Return} $(V',A',T',c',r')$.
				\end{enumerate}
			\end{transf}
			A comparison of Transformation~\ref{transf:PC} and Transformation~\ref{transf:RPC} is illustrated in Figure~\ref{fig:pc}.
			\iftrue
			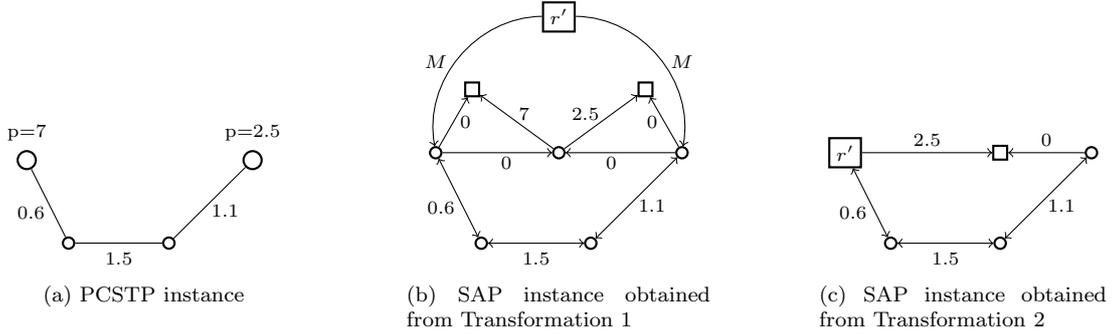
\begin{figure}[ht]
				\centering
				\subfloat[\pc instance]{	
					\begin{tikzpicture}[scale=1.1]
					\tikzstyle{terminal} = [draw, thick, minimum size=0.6, inner sep=2.6pt];
					\tikzstyle{steiner} = [circle, draw, thick, minimum size=0.2, inner sep=1.5pt];
					\tikzstyle{steinerf} = [circle, draw, thick, minimum size=0.4, inner sep=2.4pt];
					\node[steinerf, label=above:{\scriptsize p=2.5}] (t2) at (4.2,2.0) {};
					\node[steinerf, label=above:{\scriptsize  p=7}] (t3) at (1.5,2.0) {};
					
					\node[steiner] (s1) at (2.0,1.0) {};
					\node[steiner] (s3) at (3.2,1.0) {};
					
					\draw (t2) -- (s3) node [midway, below right=-4pt] {\scriptsize 1.1};		
					\draw (t3) -- (s1) node [midway, below left=-3pt] {\scriptsize 0.6};		
					\draw (s1) -- (s3) node [midway, below] {\scriptsize 1.5};
					\end{tikzpicture}

					\label{fig:pc-prob}
				}
				\hfill
				\subfloat[SAP instance obtained from Transformation~\ref{transf:PC}]{
					
					\begin{tikzpicture}[scale=1.2]
					\tikzstyle{terminal} = [draw, thick, minimum size=0.6, inner sep=2.6pt];
					\tikzstyle{terminalt} = [draw,  thick, minimum size=0.6, inner sep=2.6pt];
					\tikzstyle{steiner} = [circle, draw, thick, minimum size=0.2, inner sep=1.5pt];
					\tikzstyle{steinert} = [circle, draw, thick, minimum size=0.2, inner sep=1.5pt];
					\tikzstyle{steinerf} = [circle, draw, thick, minimum size=0.4, inner sep=2.4pt];
					\node[terminalt] (t1) at (2.85,3.5) {\scriptsize $r'$};
					\node[steinert] (v0) at (2.85,2) {};
					\node[terminalt] (t2b) at (3.8,2.7) {};
					\node[terminalt] (t3b) at (1.9,2.7) {};
					\node[steinert] (t2) at (4.2,2.0) {};
					\node[steinert] (t3) at (1.5,2.0) {};
					
					\node[steiner] (s1) at (2.0,1.0) {};
					\node[steiner] (s3) at (3.2,1.0) {};
					
					\node [] at (1.5,3.0) {\scriptsize ${M}$};
					\node [] at (4.2,3.0) {\scriptsize ${M}$};
					
					\draw[->] (t2) -- (t2b) node [midway, left=-1pt] {\scriptsize {0}};
					\draw[->] (t3) -- (t3b) node [midway, right=-1pt] {\scriptsize {0}};
					\draw[->] (t2) -- (v0) node [midway, below left=-3pt] {\scriptsize {0}};
					\draw[->] (t3) -- (v0) node [midway, below right=-3pt] {\scriptsize {0}};
					
					\draw[->] (v0) -- (t3b) node [midway, above right=-4.0pt] {\scriptsize {7}};
					\draw[->] (v0) -- (t2b) node [midway, above left=-4.0pt] {\scriptsize {2.5}};
					
					\draw[->, bend angle=50, bend right] (t1) to (t3);
					\draw[->, bend angle=50, bend left] (t1) to (t2);

					\draw[<->] (t2) -- (s3) node [midway, below right=-4pt] {\scriptsize 1.1};
					
					\draw[<->] (t3) -- (s1) node [midway, below left=-3pt] {\scriptsize 0.6};
					
					\draw[<->] (s1) -- (s3) node [midway, below] {\scriptsize 1.5};
					
					\end{tikzpicture}
					
					\label{fig:pc-trans}
				}
				\hfill
				\subfloat[SAP instance obtained from Transformation~\ref{transf:RPC}]{
					\begin{tikzpicture}[scale=1.2]
					\tikzstyle{terminal} = [draw, thick, minimum size=0.6, inner sep=2.6pt];
					\tikzstyle{terminalt} = [draw, thick, minimum size=0.6, inner sep=2.6pt];
					\tikzstyle{steiner} = [circle, draw, thick, minimum size=0.2, inner sep=1.5pt];
					\tikzstyle{steinert} = [circle, draw, thick, minimum size=0.2, inner sep=1.5pt];
					\tikzstyle{steinerf} = [circle, draw, thick, minimum size=0.4, inner sep=2.4pt];
					
					\node[terminalt] (t2b) at (3.2,2.0) {};
					\node[steinert] (t2) at (4.2,2.0) {};
					\node[terminalt] (t3) at (1.5,2.0)  {\scriptsize $r'$};
					
					\node[steiner] (s1) at (2.0,1.0) {};
					\node[steiner] (s3) at (3.2,1.0) {};

					\draw[->] (t2) -- (t2b) node [midway, above=-1pt] {\scriptsize {0}};
					\draw[->] (t3) -- (t2b) node [midway, above=-1pt] {\scriptsize {2.5}};
					
					\draw[<->] (t2) -- (s3) node [midway, below right=-4pt] {\scriptsize 1.1};
					
					\draw[<->] (t3) -- (s1) node [midway, below left=-3pt] {\scriptsize 0.6};

					\draw[<->] (s1) -- (s3) node [midway, below] {\scriptsize 1.5};
					
					\end{tikzpicture}		
					\label{fig:pcr-trans}
				}
				\caption{Illustration of a \pc instance (left) and the equivalent SAP obtained by Transformation~\ref{transf:PC} (middle).
				Given the information that the potential terminal with weight $p=7$ is part of at least one optimal solution, Transformation~\ref{transf:RPC} yields the SAP depicted on the right.  The terminals of the SAPs are drawn as squares and the (two) potential terminals for the \pc are enlarged.} 
				\label{fig:pc}
			\end{figure}
			\fi
		
			For an \rpc $(V,E,T_f,c,p,r)$ we define the \emph{transformed rooted prize-collecting cut} ($PrizeRCut$) formulation, similar to $PrizeCut$, based on the SAP instance $(V',A',T',c',r')$ obtained from Transformation~\ref{transf:RPC}: 
	\begin{formulation}{Transformed rooted prize-collecting cut (PrizeRCut)}
									\label{form:PrizeRCut}
			\begin{equation}
				\min\{ {c'}^T x \with x \text{ satisfies } \eqref{form:dcut:1},\eqref{form:dcut:2},~(x,y) \text{ satisfies }\eqref{form:PrizeCut:2},~y\text{ satisfies }    \eqref{form:PrizeCut:3}\}.
			\end{equation}
	\end{formulation}
			By $PrizeRCut(I_{RPC}, t_p, t_q)$ we denote the $PrizeRCut$ formulation for an \rpc $I_{RPC}$ when using (fixed) terminals $t_p, t_q$ in Transformation~\ref{transf:RPC}.
			One may wonder whether the choice of $t_p$ and $t_q$ affects $v_{LP}(PrizeRCut(I_{RPC}, t_p, t_q))$; in fact, it does not, and even more:
			\begin{proposition}
				\label{prop:roots}
				Let $I_{RPC}$ be an \rpc and let $t_p, t_q, t_{\tilde{p}}, t_{\tilde{q}}$ be any of its fixed terminals. Define $R(t_i,t_j) := \mathcal{P}_{LP}(PrizeRCut(I_{RPC}, t_i, t_j))$. It holds that:
				\begin{equation}
				proj_y(R(t_p,t_q)) = proj_y(R(t_{\tilde{p}}, t_{\tilde{q}})).
				\end{equation} 
			\end{proposition}
	\begin{proof}
					Let $(V,E,T_f,c,p)$ be the \rpc $I_{RPC}$ and denote the SAP resulting from applying Transformation~\ref{transf:RPC} on $(I_{RPC}, t_p, t_q)$ by $(V',A',T',c',t_q)$. Set $D = (V',A')$.
	                Furthermore, let $x,y$ be a feasible solution to the LP relaxation of $PrizeRCut(I_{RPC}, t_p,t_q)$---so $(x,y) \in R(t_p, t_q)$. For ease of presentation, we will use the notation $x_{ij}$ instead of $x((v_i,v_j))$ for an arc $(v_i,v_j)$.
	                	The proposition will be proved in two steps: first by fixing $t_q$ and changing $t_p$, and second by fixing $t_p$ and changing $t_q$. Note that due to symmetry reasons in both cases it is sufficient to show that one projection is contained in the other.
				\paragraph{1) $proj_y(R(t_p, t_q)) = proj_y(R(t_{\tilde{p}}, t_q))$}
					Let $\tilde{I}_{\tilde{p}} = (\tilde{V}, \tilde{A}, \tilde{T},\tilde{c}, t_{{q}})$ be the SAP resulting from applying Transformation~\ref{transf:RPC} on $(I_{RPC}, t_{\tilde{p}}, t_{{q}})$, 
					and set $\tilde{D} := (\tilde{V}, \tilde{A})$; note that $\tilde{V} = V'$ and $\tilde{T} = T'$. 
					Define $\xp \in [0,1]^{\tilde{A}}$ by $\xp((t_{\tilde{p}}, t_i')) := x((t_p, t_i'))$ for $i=1,...,z$ (with the notation from Transformation~\ref{transf:RPC}) and by $\xp_{ij} := {x}_{ij}$ for all remaining arcs. 
					Suppose that there is a $U \subset \tilde{V}$ with $t_q \notin U$ and $U \cap \tilde{T} \neq \emptyset$ such that $\xp(\delta^-_{\tilde{D}}(U)) < 1$.
					 From $x(\delta^-_{D}(U)) \geq 1$ and the construction of $\xp$ it follows that 
					 $t_{\tilde{p}} \in U$---otherwise $\xp(\delta^-_{\tilde{D}}(U)) \geq x(\delta^-_{D}(U))$. 
					 For $U^z := U \setminus \{t_1',...,t_z'\}$ one obtains
					 \begin{equation}
					 \label{proof:roots:1}
					x(\delta^{-}_{D}(U^z)) = \xp(\delta^{-}_{\tilde{D}}(U^z)) \leq \xp(\delta^{-}_{\tilde{D}}(U)) < 1. 
					 \end{equation}
					Because of $t_q \notin U^z$ and $U^z \cap \tilde{T} \supseteq \{t_{\tilde{p}}\} \neq \emptyset$, 
					one obtains a contradiction from~\eqref{proof:roots:1}. Therefore, $\xp$ satisfies~\eqref{form:dcut:1} for the SAP $\tilde{I}_{\tilde{p}}$. Furthermore, $\yp$ defined by $\yp(\{v_i,v_j\}) := \xp_{ij} + \xp_{ji}$ for all $\{v_i,v_j\} \in E$ satisfies $\yp = y$. 
				\paragraph{2) $proj_y(R(t_p, t_q)) = proj_y(R(t_p, t_{\tilde{q}}))$}
				Define the SAP $\tilde{I}_{\tilde{q}} := (V', A', T', c', t_{\tilde{q}})$ (the result of transforming $(I_{RPC}, t_p, t_{\tilde{q}})$). 
				As there is only one underlying directed graph (namely $D$), in the following we write $\delta^-$ instead of $\delta^-_{D}$.
					Let $f$ be a 1-unit flow from $t_q$ to $t_{\tilde{q}}$ such that $f_{ij} \leq x_{ij}$ for all $(v_i,v_j) \in A'$. Define $\xq$ by $\xq_{ij} := x_{ij} + f_{ji} - f_{ij}$ for all $(v_i,v_j) \in A'$. Let $U \subset V'$ such that $t_{\tilde{q}} \notin U$ and $U \cap T' \neq \emptyset$. If $t_q \notin U$, then $f(\delta^-(U)) = f(\delta^+(U))$ and so $\xq(\delta^-(U)) = x(\delta^-(U)) \geq 1$. On the other hand, if $t_q \in U$, then $f(\delta^+(U)) = f(\delta^-(U)) + 1$, so
					\begin{equation}
					\xq(\delta^-(U)) \geq x(\delta^-(U)) + 1\geq 1.  
					\end{equation}
					Consequently, $\xq$ satisfies~\eqref{form:dcut:1} for the SAP $\tilde{I}_{\tilde{q}}$.
					From $x_{ij} + x_{ji} \leq 1$ for all $(v_i,v_j) \in A'$, it follows that $\xq \in [0,1]^{{A}'}$, and for the corresponding $\yq$ one verifies $\yq = y$.
					\end{proof}
	       		Consequently, if only the $y$ variables are of interest, we write $PrizeRCut(I_{RPC})$ instead of $PrizeRCut(I_{RPC}, t_p,t_q)$. For the (heuristic) dual-ascent algorithm the choice of $t_p$ and $t_q$ still matters, as it can change both lower bound and reduced costs. Therefore, we repeat the dual-ascent reduction techniques~\cite{RehfeldtKochMaher16} on several SAPs resulting from different choices of $t_p$ and $t_q$.
			
		From the definitions of Transformation~\ref{transf:PC} and~\ref{transf:RPC} one can acknowledge that switching from $PrizeCut$ to $PrizeRCut$ (if possible) does not deteriorate (and can improve) the tightness of the LP relaxation; due to its importance we formally state this observation:
			\begin{lemma}
					\label{lemma:rpc0}
					For $\ppc =(V,E,c,p)$ let $T_0 \subseteq T_p$ such that $T_0 \subseteq V(S)$ for at least one optimal solution $S$ to $\ppc$. Let $I_{T_0} := (V,E,T_0,c,p)$ be an $\rpc$. With $R_{T_0} := \mathcal{P}_{LP}(PrizeRCut(I_{T_0}))$, $R := \mathcal{P}_{LP}(PrizeCut(\ppc))$ it holds that
				\begin{equation}
					proj_y( R_{T_0} )\subseteq proj_y(R).
				\end{equation}
			\end{lemma}
			\begin{proposition}
			\label{prop:rpc1}
			With $T_0$ and $I_{T_0}$ defined as in Lemma~\ref{lemma:rpc0} the inequality
			\begin{equation}
		 v_{LP}(PrizeCut(\ppc)) \leq v_{LP}(PrizeRCut(I_{T_0}))
			\end{equation}
			holds and can be strict.
			\end{proposition}
			A proof of Proposition~\ref{prop:rpc1} is given in Appendix~\ref{app:rpc1}.

			Finally, by combining the reductions to \rpc and SAP with the reductions techniques described in Section~\ref{sec:reduction}, it is sometimes possible to either eliminate or fix each potential terminal. Hence the instance becomes an \stp, which allows to apply a number of further algorithmic techniques~\cite{SCIPJACKMPC,Polzin04}.

			\section{Reduction-based exact solving}
			\label{sec:computational}
			This section describes how the new algorithms are integrated within a branch-and-cut framework. Also, the performance of the resulting solver is discussed.
			
			\subsection{Interleaving the components within branch-and-cut}
			\label{sec:full}
			The exact solver described in this article is realized within the branch-and-cut based Steiner tree framework \scipjack~\cite{SCIPJACKMPC}. 
			\scipjack already includes reduction techniques for \pc~\cite{RehfeldtKochMaher16} (in the sense of Section~\ref{sec:reduction}), but almost all of them have been replaced by new methods introduced in this article. Furthermore, we use the reduction techniques for domain propagation, translating the deletion of edges and the fixing of potential terminals into variable fixings in the IP. 
			Additionally, we employ a technique similar to the probing~\cite{Savelsbergh94} approach for general MIPs: Instead of setting binary variables to $0$ or $1$, we fix or delete potential terminals. By using the left-rooted prize-constrained distance, in each case it is often possible to either fix or delete additional potential terminals---which usually allows for further graph reductions.
		
			\scipjack also includes several (generic) primal heuristics that can be applied for \pc. Most compute new solutions on newly built subgraphs (e.g. by merging feasible solutions). For these heuristics the new reduction techniques can often increase the solution quality. In turn, an improved upper bound can allow for further graph reductions (e.g. by the terminal-regions composition) or to fix additional terminals (by means of Proposition~\ref{prop:imp}).
			 Additionally, we have implemented a new primal heuristic that starts with a single (potential or fixed) terminal and connects other terminals $t_i$ to the current subtree $S$ if $\min_{v \in V(S)} d^-_{pc}(t_i,v) \leq p(t_i)$ (only upper bounds on $d^-_{pc}$ are used).
			 A similar approach has been implemented as a local search heuristic.

		
		 Also the LP kernel interacts with the remaining components: By means of the prize-constrained distances and upper bounds provided by the heuristics it is usually possible to switch to the $PrizeRCut$ formulation. In turn, the reduced costs and lower bound provided by an improved LP solution can be used to reduce the problem size~\cite{RehfeldtKochMaher16}---which can even enable further prize-constrained walk based reductions. Moreover, besides the separation of~\eqref{form:dcut:1}, already implemented in \scipjack, we also separate constraints for $TransRCut$ of the form
		 \begin{equation*}
		 	\label{cons:impl2}
		 	  x(\delta^-(v_j)) +  x((t_p, t_i')) \leq 1  ~~~~~~ t_i \in T_p \setminus T_f, v_j\in \{ v \in V \with d^-_{pc}(t_i,v) < p(t_i) \}.
		 \end{equation*}
		 The constraints represent the implication that $v \in  V(S) \Rightarrow t \in V(S)$ for any optimal solution $S$ if $d^-_{pc}(t,v) < p(t)$. Corresponding constraints are separated for $TransCut$.
		
			Finally, branching is performed on vertices---by rendering vertex $v_i$ to branch on a fixed terminal (and transforming the problem to \rpc if not already done) in one branch-and-bound child node and removing it in the other. As in probing, the implications from the left-rooted prize-constrained distance often allow further graph changes.  
		   Throughout the solving process, we switch to the \stp solver of \scipjack~\cite{SCIPJACKMPC} if all potential terminals could be fixed.
		
			
			\subsection{Computational results}
			\label{sec:ane:com}
			To the best of our knowledge, the three strongest exact algorithms for \pc are from~\cite{Fischetti2017,SCIPJACKMPC,Leitner18}. While no solver dominates on all benchmark test sets, the  branch-and-bound solver from~\cite{Leitner18} is competitive on almost all, and on several ones orders of magnitude faster---it is even faster than state-of-the-art heuristic methods~\cite{FuH17}. Thus it will in the following be used for comparison.

				 \begin{table}[ht]
				 	\centering
				 	\scriptsize
				 	\caption{Details on \pc tests sets.}
				 	\label{tab:instances}
				 	\begin{tabular*}{1.0\textwidth}{@{\extracolsep{\fill}}lcccll@{}}
				 		Name & Instances & $|V|$ & $|E|$ & Status & Description \\
				 		\midrule
				 		JMP & 34 &   100 - 400  & 315 - 1576  &  solved & Sparse instances of varying structure, \\
				 		 &  &      &     & &   introduced in~\cite{Joh00}. \\
				 		    Cologne1 & 14 &   741 - 751  & 6332 - 6343  & solved~ \rdelim\}{3}{10pt} & \multirow{2}{*}{Instances  derived from the design of fiber }\\
				 		     &  &      &                           &             & \multirow{2}{*}{optic networks for German cities~\cite{IvanaL2004}}.\\
				 		    Cologne2 & 15 &   1801 - 1810  & 16719 - 16794  &      solved       & \\
				 		    &&&&&\\
				CRR & 80 &   500 - 1000  &  25000  & solved & Mostly sparse instances, based on \\
				 &  &      &      &&    C and D test sets of the SteinLib~\cite{IvanaL2004}. \\ 	
				 E & 40 &   2500  &  3125 - 62500  & solved & Mostly sparse instances originally for  \\
				  &  &      &      &&    \stp, introduced in~\cite{IvanaL2004}. \\	 
				ACTMOD & 8 &  2034 - 5226  & 3335 - 93394 & solved &Real-world instances  derived  from \\
				&  &      &   &  &    integrative biological network analysis~\cite{Dit08}. \\
				HANDBI & 14 &  158400  &    315808 & unsolved ~\rdelim\}{3}{10pt} & \multirow{2}{*}{Images of hand-written text derived}\\
								&  &      &     & &    \multirow{2}{*} {from a signal processing problem~\cite{DimacsWeb}. } \\
				HANDBD & 14 &  169800  & 338551 &  unsolved &  \\
				&  &      &     & &   \\
				PUCNU & 18 &   64 - 4096  &  192 - 28512 &  unsolved & Hard instances derived from the PUC set  \\
				 &  &      &      & &     for \stp. From 11th DIMACS Challenge.  \\
				 	H & 14 &   64 - 4096  &  192 - 24576 &  unsolved & Hard instances based on hypercubes.  \\
				 				 &  &      &      & &      From 11th DIMACS Challenge.  \\
				 		\bottomrule
				 	\end{tabular*}
				 \end{table}
	
			 The computational experiments were performed on Intel Xeon CPUs E3-1245 with 3.40~GHz and 32 GB RAM. For our approach CPLEX 12.7.1\footnote{http://www-01.ibm.com/software/commerce/optimization/cplex-optimizer/} is employed as underlying \LP solver---\cite{Leitner18} does not use an \LP solver. Furthermore, only single-thread mode was used (as~\cite{Leitner18} does not support multiple threads.
		For the following experiments 11 benchmark test sets from the literature and the 11th DIMACS Challenge are used, as detailed in Table~\ref{tab:instances}.
	
			 Table~\ref{tab:resultsavg} provides aggregated results of the experiments with a time limit of one hour. The first column shows the test set considered in the current row. Columns two and three show the shifted geometric mean~\cite{Achterberg07a} (with shift $1$) of the run time taken by the respective solvers. The next two columns provide the maximum run time, the last two columns the number of solved instances.  
			\begin{table}[ht]
				\centering
				\footnotesize
				\caption{Computational comparison of the solver described in~\cite{Leitner18}, denoted by \textit{\cite{Leitner18}}, and the solver described in this article, denoted by \textit{new}.}
							\label{tab:resultsavg}
				\begin{tabular*}{0.9\textwidth}{@{\extracolsep{\fill}}llrrrrrr@{}}
					\toprule
					& &       \multicolumn{2}{c}{mean time\,[s]} & \multicolumn{2}{c}{max. time\,[s]} & \multicolumn{2}{c}{\# solved} \\
					\cmidrule(lr){3-4} \cmidrule(lr){5-6}   \cmidrule(lr){7-8}
					Test set  & \# &  ~~\cite{Leitner18} &  new&~~~\cite{Leitner18} & new &~~~\cite{Leitner18} &  new  \\
					\midrule
					JMP   &   34    &  0.0            & 0.0 & 0.0 & 0.0 & 34 & 34 \\
					Cologne1   &  14 &    0.0            & 0.0 & 0.1 & \textbf{0.0} & 14 & 14 \\
					Cologne2   &  15 &     0.1            & 0.1 & 0.2 & \textbf{0.1} & 15  & 15 \\
					CRR   &   80  &   0.1            & 0.1 & 5.7 & \textbf{1.1} & 80 & 80 \\
					ACTMOD  &     8 & 0.9            & \textbf{0.3} & 3.5 & \textbf{1.5} & 8 & 8 \\
					E   &         40 & 1.8           & \textbf{0.2} & $>$3600 & \textbf{34.5} & 37 & \textbf{40} \\				
					HANDBI  &     14 & 36.5            & \textbf{14.9} & $>$3600 & $>$3600 & 12 & \textbf{13} \\
					HANDBD  &     14 & 34.1            & \textbf{17.9} & $>$3600 & $>$3600 & 13 & 13 \\
					I640   &      100 & 8.7            & \textbf{6.1} &  $>$3600 &  $>$3600 & 90 &  \textbf{91} \\
					PUCNU   &     18 & 278.9            & \textbf{80.2} & $>$3600 & $>$3600 & 7 & \textbf{11} \\
					H   &         14 & 488.7            & \textbf{477.4} & $>$3600 &$>$3600 & 4 & \textbf{5} \\
					\bottomrule
				\end{tabular*}
			\end{table}
		
		 The new solver is on each test set faster than or as fast as~\cite{Leitner18}, both in terms of the maximum and average run time. While both solvers can solve all instances from JMP, CRR, Cologne1, Cologne2, and ACTMOD to optimality,
		  on all remaining test sets except HANDBD the new approach solves more instances than~\cite{Leitner18}.
		  Moreover, it solves all instances solved by~\cite{Leitner18}, and the primal-dual gap on each instance that cannot be solved by either solver is smaller than that of~\cite{Leitner18}---usually by a factor of more than $2$.
		  Furthermore, the new solver improves the best known upper bounds for more than a third of the previously unsolved instances from the DIMACS Challenge, with two being solved to optimality, as detailed in~\ref{app:impr}.
	

			\iftrue
			\section{Acknowledgements}
			
			This work was supported by the BMWi project \emph{Realisierung von Beschleunigungsstrategien der anwendungsorientierten Mathematik und Informatik f\"ur optimierende Energiesystemmodelle - BEAM-ME} 
			(fund number 03ET4023DE).
			The work for this article has been conducted within the Research Campus Modal funded by the German Federal Ministry of Education and Research (fund number 05M14ZAM).
			\fi
					\bibliographystyle{plain}
			\bibliography{pcstpZIB}
			\newpage
			\begin{appendix}
				\section{Results for unsolved DIMACS instances}
			\label{app:impr}
			Using an extended time limit of four hours, we could improve or even solve more than one third of the (previously)  unsolved \pc instances from the 11th DIMACS Challenge. 
			All improved instances are listed in Table~\ref{tab:best}, with the first column giving the name of the instance, the second its primal-dual gap, the third the improved found bound, and the fourth the previously best known one.\footnote{The solution of the instance \emph{cc7-3nu} is already noted in the report~\cite{SCIP6}, but is based on techniques described in this article (which have not been published before).} 
			\begin{table}
			\centering
						\footnotesize
			\caption{Improvements on unsolved DIMACS instances.}\label{tab:best}
			\begin{tabular*}{0.8\textwidth}{@{\extracolsep{\fill}}lrrr@{}}
			\hline
			Name &  gap [\%] & new UB & previous UB \\
			\hline
				cc7-3nu    & \textbf{opt} &     \textbf{270} & 271 \\
				cc10-2nu  & \textbf{opt} &     \textbf{167} & 168 \\
				cc11-2nu  & 1.2 &     304 & 305 \\
				cc12-2nu  & 1.0 &     565 & 568 \\
		hc9p2       & 1.4          &            30230   & 30242 \\  
		hc10p       & 1.4          &            59778   & 59866 \\   
		hc10p2       & 1.5          &           59807   & 59930 \\     
		hc11p       & 1.6          &            118729   & 119191 \\    
		hc11p2       & 1.8          &           118979   & 119236 \\ 
				    i640-341  & 0.5 &     29691 & 29700  \\
				    i640-344  & 0.6 &     29910 &  29921 \\   		  
			\hline
			\end{tabular*}
			\end{table}
				\section{Further proofs}

				\subsection{Proof of Proposition~\ref{prop:bnd1}} 	
				\label{app:bnd}
				\begin{proof}
				Initially, define $b: V \to T_p$ such that $v \in H_{b(v)}$ for all $v \in V$.
				Assume that there exists an optimal solution $S$ such that $v_i \in V(S)$.
				  Denote the (unique) path in $S$ between $v_i$ and any $t_j \in
				  V(S) \cap T_p$ by $Q_j$ and the set of all such paths by $\mathcal{Q}$.
				  First, note that $|\mathcal{Q}| \geq 2$, because 
				  if $\mathcal{Q}$ just contained one path, say $Q_l$, the single-vertex tree $\{t_l\}$ would be of smaller cost
				  than $S$.
				  Second, if a vertex $v_k$ is contained in two distinct paths in $\mathcal{Q}$, the subpaths of these two paths between $v_i$ and $v_k$ coincide. Otherwise there would need to be a cycle in $S$. Additionally, there are at least two paths in $\mathcal{Q}$ having only the vertex $v_i$ in common. Otherwise, due to the precedent observation, all paths would have one edge $\{v_i, v_i'\}$ in common, which could be discarded to yield a tree of smaller cost than $C(S)$.
				
				  Let $Q_k \in \mathcal{Q}$ and $Q_l \in \mathcal{Q}$ be two distinct paths with $V(Q_k) \cap V(Q_l) = \{v_i\}$
				  such that
				  \begin{equation}
				  \label{proof:bnd:1}
				   | \{ \{v_x, v_y\} \in E(Q_k) \cup E(Q_l) \with b(v_x) \neq b(v_y) \} |
				  \end{equation}
				   is minimized. Define $\mathcal{Q^-} := \mathcal{Q} \setminus \{Q_k, Q_l\}$. For all $Q_r \in \mathcal{Q^-}$, denote by
				  $Q'_r$ the subpath of $Q_r$ between $t_r$ and the first vertex not in
				  $H_{t_r}$. Suppose that $Q_k$ has an edge $e \in E(S)$ in common with a $Q'_r$:
				  Consequently, $Q_l$ cannot have any edge in common with $Q_r$, because this would
				  require a cycle in $S$. Furthermore, $Q_k$ and $Q_r$ have to contain a joint subpath including $v_i$ and
				  $e$. But this would imply that $Q_k$ contained at least one additional edge  $\{v_x, v_y\}$ with $b(v_x) \neq b(v_y)$. Thus, $Q_r$ would have initially been selected instead of $Q_k$.
				
				  Following the same line of argumentation, one validates that
				   $Q_l$ has no edge in common with any $Q'_r$.
				  Conclusively, the paths $Q_k$, $Q_l$ and all $Q'_r$ are edge-disjoint.
				  Thus one obtains:
				  \begin{align*}
				    C(S) &= \sum_{e \in E(S)} c(e) + \sum_{v \in V \setminus V(S)} p(v) \\
				    &\geq \bigg(\sum_{Q_r \in \mathcal{Q^-} } c(E(Q'_r)) \bigg) +
				    c(E(Q_k)) + c(E(Q_l)) + \sum_{v \in V \setminus V(S)} p(v) \\
				    &\geq \sum_{q = 1}^{s - 2} r_H^{pc}(t_q) + c(E(Q_k)) + c(E(Q_l)) \\
				    &\geq \sum_{q = 1}^{s - 2} r_H^{pc}(t_q) + \underline{d}(\uv_i, \uv_{i,1}) + \underline{d}(\uv_i, \uv_{i,2}),
				  \end{align*}
				  which proves the proposition.
				\end{proof}

				\iftrue

				\subsection{Proof of Proposition~\ref{proposition:pvdnp}}
				\label{app:pvd}
				We will show the $\mathcal{NP}$-hardness already for the \stp variant of the \trd (which implies the $\mathcal{NP}$-hardness for \pc). For an \stp $(V,E,T,c)$ define the \trd as a partition $H = \big\{ H_{t} \subseteq V \with T \cap H_{t} = \{t\}\big \}$ of $V$ such that for each $t \in T$ the subgraph induced by $H_{t}$ is connected. 
								Define for all $t \in T$
								\begin{equation}
									\label{bnd2}
									r_H(t) :=  \min\{ d(t,v) \with v \notin H_{t} \}.
								\end{equation}
								Note that this definition is just a special case of the \pc version (for \pc instances with sufficiently high vertex weights).
				First, the decision variant of the \trd problem is stated. Let $\alpha \in \mathbb{N}_0$ and let $G_0 = (V_0, E_0)$ be an undirected, connected graph with  costs $c: E \rightarrow
						\mathbb{N}$. Furthermore, set $T_0 := \{ v \in V_0 \with p(v) > 0\}$, and assume that $\alpha < |T_0|$. For each \trd $H_0$ of $G_0$ define $T'_0 \subsetneq T_0$ such that $|T'_0| = \alpha$ and $r_{H_0}(t') \geq r_{H_0}(t)$ for all $t' \in T'_0$ and  $t \in T_0 \setminus T'_0$. Let $
							C_{H_0} := \sum_{t \in T_0 \setminus T_0'} r_{H_0}(t).$
						We now define the $\alpha$ \trd problem as follows: Given a $k \in \mathbb{N}$, is there a \trd $H_0$ such that $C_{H_0} \geq k$? The next lemma forthwith establishes the $\mathcal{NP}$-hardness of finding a \trd that maximizes~\eqref{prop:bnd1:1}, or~\eqref{prop:bnd3:1}---which corresponds to $\alpha = 2$ and $\alpha = 3$, respectively.
						\begin{lemma}
							\label{proposition:pvd}
							For each $\alpha \in  \mathbb{N}_0$ the $\alpha$ \trd problem is $\mathcal{NP}$-complete.
						\end{lemma} 
			\begin{proof}
				Given a \trd $H_0$ it can be tested in polynomial time whether $C_{H_0} \geq k$. 
				Consequently, the \trd problem is in $\mathcal{NP}$. 
		
				In the remainder it will be shown that the ($\mathcal{NP}$-complete~\cite{Garey1979}) independent set problem can be reduced to the \trd problem.
				To this end, let $G_{ind} = (V_{ind}, E_{ind})$ be an undirected, connected graph and $k \in \mathbb{N}$. The problem is to determine whether an independent set in $G_{ind}$ of cardinality at least $k$ exists.
				To establish the reduction, construct a graph $G_0$ from $G_{ind}$ as follows.
				Initially, set $G_0 = (V_0,E_0) := G_{ind}$. Next, extend $G_0$ by replacing each edge $e_l = \{v_i,v_j \} \in E_{0}$ with a vertex $v_l'$ and the two edges $\{v_i,v_l'\}$ and $\{v_j,v_l'\}$.
				Define edge weights $c_0(e) = 1$ for all $e \in E_0$ (which includes the newly added edges).
				If $\alpha > 0$, choose an arbitrary $v_i \in V_0 \cap V_{ind}$ and add for $j = 1,...,\alpha$ vertices $t_i^{(j)}$ to both $V_0$ and $T_0$. Finally, add for $j = 1,...,\alpha$ edges $\{v_i, t_i^{(j)}\}$ with $c_0(\{v_i, t_i^{(j)}\}) = 2$ to $E_0$. 
				
				First, one observes that the size $|V_0| + |E_0|$ of the new graph $G_0$ is a polynomial in the size $|V_{ind}| + |E_{ind}|$ of $G_{ind}$. Next, $r_{H_0}(v_i) = 2$ holds for a vertex $v_i \in G_0 \cap G_{ind}$ if and only if $H_{v_i}$ contains all (newly inserted) adjacent vertices of $v_i$ in $G_0$.
				Moreover, in any \trd $H_0$ for $(G_0, c_0)$, it holds that $r_{H_0}(t_i^{(j)}) = 2$ for $j = 1,...,\alpha$. Hence, there is an independent set in $G_{ind}$ of cardinality at least $k$ if and only if there is a \trd $H_0$ for $(V_0, E_0, T_0, c_0)$ such that 
				\begin{align*}
				C_{H_0} \geq |V_{ind}| + k
				\end{align*}
				This proves the proposition.
			\end{proof}

			 \fi

				\subsection{Proof of Proposition~\ref{prop:rpc1}}
				\label{app:rpc1}
				\begin{proof}
				First it follows from the construction of Transformation~\ref{transf:PC} and ~\ref{transf:RPC} that each optimal solution $x^0,y^0$ to the LP relaxation of $TransRCut(I_{T_0})$ can be transformed to a solution $x,y$ to the LP relaxation of $TransCut(\ppc)$ without changing the objective value: By setting $x((v_i,v_j)) := x^0((v_i,v_j))$ and $x((v_j,v_i)) := x^0((v_j,v_i))$ for all $\{v_i,v_j\} \in E$, $x((r', t_0)) := 1$ for any $t_0 \in T_0$,  $x((t_i, t_i')) := 1$ for all $t_i \in T_0$, and by setting the remaining $x((v_i,v_j))$ accordingly. Thus $v_{LP}(PrizeCut(\ppc)) \leq v_{LP}(PrizeRCut(I_{T_0}))$. To see that the inequality can be strict, consider the following wheel instance (which is well-known to have an integrality gap for DCut on \stp):\\~
				\newline
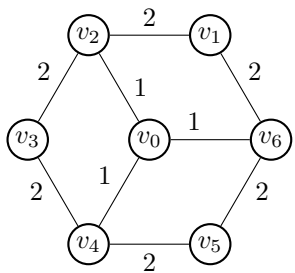
						\begin{tikzpicture}[scale=0.8]
						\tikzstyle{steiner} = [circle, draw, thick, minimum size=0.2, inner sep=1.5pt];
					    \node[steiner,fill=white] (v_0) at (0,0) {$v_0$};
					\foreach \phi in {1,...,6}{
					    \node[steiner,fill=white] (v_\phi) at (360/6 * \phi:2cm) {$v_\phi$};
					      }
						    \draw (v_0) -- (v_2) node [midway, right=2pt] {1};	
					      	\draw (v_0) -- (v_6) node [midway, above left] {1};
						    \draw (v_0) -- (v_4) node [midway, above left=-1pt] {1};	
					      	\draw (v_2) -- (v_3) node [midway, above left=-1pt] {2};
						    \draw (v_3) -- (v_4) node [midway, left=2pt] {2};	
					      	\draw (v_4) -- (v_5) node [midway, below] {2};
						    \draw (v_5) -- (v_6) node [midway, right=2pt] {2};	
					      	\draw (v_1) -- (v_2) node [midway, above] {2};	
							\draw (v_1) -- (v_6) node [midway, above right=-1pt] {2};			
										\end{tikzpicture}\\	
					Set $p(v_0) = p(v_1) = p(v_3) = p(v_5) = 4$, $p(v_2) = p(v_6) = 0$, and $p(v_4) = \epsilon$ with $0 < \epsilon < 1$. Let $T_0 := \{v_0, v_1, v_3, v_4, v_5\}$. Let $I$ be the \pc and $I_{T_0}$ the corresponding \rpc. It holds that $v_{LP}(TransCut(I))= 7.5 + \dfrac{\epsilon}{2} < 8 = v_{LP}(TransRCut(I_{T_0}))$. Part of the solution corresponding to $v_{LP}(TransCut(I))$ is shown below (with numbers next to the arcs denoting the $x$ values), the remaining $x$ and $y$ are set accordingly (e.g., $x((r', v_1)) = 1$).
					
						\begin{tikzpicture}[scale=0.8]
											\tikzstyle{steiner} = [circle, draw, thick, minimum size=0.2, inner sep=1.5pt];
										    \node[steiner,fill=white] (v_0) at (0,0) {$v_0$};
										\foreach \phi in {1,...,6}{
										    \node[steiner,fill=white] (v_\phi) at (360/6 * \phi:2cm) {$v_\phi$};
										      }
			    \draw[<-, thick] (v_0) -- (v_2) node [midway, right=2pt] {0.5};	
		      	\draw[<-, thick] (v_0) -- (v_6) node [midway, above] {0.5};
			    \draw[->, thick] (v_0) -- (v_4) node [midway, above left=-1pt] {0.5};	
		      	\draw[->, thick] (v_2) -- (v_3) node [midway, above left=-1pt] {0.5};
			    \draw[<-, thick] (v_3) -- (v_4) node [midway, left=2pt] {0.5};	
		      	\draw[->, thick] (v_4) -- (v_5) node [midway, below] {0.5};
			    \draw[<-, thick] (v_5) -- (v_6) node [midway, right=2pt] {0.5};	
		      	\draw[->, thick] (v_1) -- (v_2) node [midway, above] {0.5};	
				\draw[->, thick] (v_1) -- (v_6) node [midway, above right=-1pt] {0.5};			
															\end{tikzpicture}
				\end{proof}
			\end{appendix}

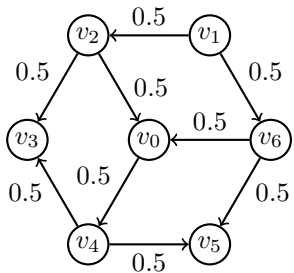
\end{document}